\documentclass{article}
\usepackage[utf8]{inputenc}
\usepackage{ amssymb }
\usepackage{ amsmath }
\usepackage{amsthm}
\usepackage{cleveref}
\usepackage{ytableau}
\usepackage{fancyhdr}
\fancypagestyle{footer}{\fancyhf{}\fancyfoot[L]{
L. J. Jolliffe, Department of Pure Mathematics and Mathematical Statistics, Centre for Mathematical Sciences, University of Cambridge, Wilberforce Road, Cambridge, CB3 0WB, United Kingdom\\
\emph{E-mail address:}  ljj33@cam.ac.uk}}
\newtheorem{theorem}{Theorem}

\newtheorem{prop}[theorem]{Proposition}

\crefalias{prop}{proposition}
\crefalias{cor}{corollary}
\newcommand{\Sn}{\mathcal{S}_n }
\newcommand{\Sm}{\mathcal{S}_m }
\newcommand{\Pin}{P_i^n(m) }
\newcommand{\Pinj}{P_i^n(m)_j }
\ytableausetup{boxsize=3em}

\title{A Short Proof of the Rank Formula for Inclusion Matrices using the Representation Theory of the Symmetric Group}
\author{Liam Jolliffe}

\date{}

\begin{document}

\maketitle
\begin{abstract}
We present a new proof of the well known formula for the rank of the inclusion matrix by constructing a $k\Sn$-module spanned by the columns of this matrix and calculating its dimension.
\end{abstract}

\section{Introduction}
The \textit{inclusion matrix,} $A_{i}^n(m)$, where $i\le n \le m$, is  the ${m \choose i} \times {m \choose n}$ matrix whose rows are indexed by subsets of $[m]:=\{1,2,\dots,m\}$ of size $i$ and whose columns are indexed by subsets of $[m]$ of size $n$. The entry corresponding to position $X,Y$ is $1$ if $X\subseteq Y$ and $0$ otherwise. This matrix arises in a number of combinatorial investigations. Gottlieb proved that over a field of characteristic $0$ this matrix has full rank \cite{Gottlieb}. Linial and Rothschild then determined a formula for the rank of this matrix over the field of two elements, as well the special case when $n=i+1$ over the field of three elements \cite{Linial}. Wilson solved the problem over any field by proving the following \cite{Wilson}:
\begin{theorem}\label{Rank}
Let $k$ be a field of characteristic $p$ and suppose $i \le \emph{min}\{n,m-n\}$. Then
$$rank_k(A_i^n(m))=\sum_{p \nmid {n-j\choose i-j}}{m \choose j}-{m \choose j-1},$$ 
where ${m \choose -1}$ is interpreted as 0.
\end{theorem}
Wilson also gives a characterisation of those vectors which are in the $\mathbb{Z}$-span of the columns of $A_i^n(m)$. Another proof of \Cref{Rank} is given by Frankl \cite{Frankl}. Observe that there is nothing lost by the assumption that $i \le \text{min}\{n,m-n\}$, because $A_i^n(m)^T=A_{m-n}^{m-i}(m)$, and so this assumption shall be made throughout. 

We shall give a new proof of \Cref{Rank} by constructing a $k\Sn$-module spanned by the columns of $A_i^{n}(m)$ which, of course, has dimension $\text{rank}_k(A_i^n(m))$. This proof shall make use of the representation theory of the symmetric group, which we review in the next section. The reader is referred to James' book \cite{James}, from which our notation is taken, for more details.

\section{Representation Theory of $\Sn$}

Recall, a \textit{partition} of an integer, $n>0$, is a non-increasing sequence of positive integers $\lambda=(\lambda_1,\lambda_2,\dots,\lambda_r)$ with $\sum_{i\ge 1}\lambda_i=n$. If $\lambda$ is a partition of $n$ we write $\lambda\vdash n$ and identify $\lambda$ with its corresponding \textit{Young diagram}: a left justified array of boxes with $\lambda_i$ boxes in the $i$th row. Partitions of $n$ index a number of important classes of $k\Sn$-modules; we shall now describe two such classes of modules: the permutation modules $M^\lambda$ and the Specht modules $S^\lambda$. To do so we will need some more combinatorial definitions.

Given a partition $\lambda$, we define a $\lambda$-\textit{tableau} to be a bijection between the set $[n]$ and the boxes of (the Young diagram of) $\lambda$. We can define an equivalence relation $\sim_r$ on the set of $\lambda$-tableaux by calling $t$ and $s$ \textit{row-equivalent} if the set of elements which appear in each row of $t$ and $s$ are the same. For example; the following two $(4,3,3,1)$-tableaux are row equivalent.

$$
\ytableausetup
{centertableaux,boxsize=1.2em}
\begin{ytableau}
1&2&3&4\\
5&6&7\\
8&9&10\\
11
\end{ytableau}
\sim_r
\begin{ytableau}
2&4&3&1\\
6&5&7\\
9&10&8\\
11
\end{ytableau}
$$
We call an equivalence class of $\lambda$-tableaux a $\lambda$-\textit{tabloid}, and denote the tabloid corresponding to the tableaux $t$ by $\{t\}$, or by drawing the Young diagram without the vertical lines, 
$$
\ytableausetup
{centertableaux,tabloids,boxsize=1.2em}
\begin{ytableau}
1&2&3&4\\
5&6&7\\
8&9&10\\
11
\end{ytableau}
.$$
There is an obvious action of $\Sn$ on both the sets of $\lambda$-tableaux and $\lambda$-tabloids, obtained by permuting the positions in which elements appear. Thus, the vector space consisting of formal sums of $\lambda$-tabloids is a $k\Sn$-module, which we denote $M^\lambda$. Throughout this paper we shall only be interested in such a module when $\lambda$ is a t\textit{wo-part partition}, that is $\lambda=(\lambda_1,\lambda_2)$, in which case we can identify the tabloid $\{t\}$ with the set of elements appearing in the second row of $\{t\}$. We then see that $M^{(m-i,i)}$ has a basis consisting of all the $i$ subsets $X\subseteq_i [m]$, that is subsets of size $i$. We shall frequently alternate between these two notations depending on notational convenience.

An important submodule of this permutation module $M^\lambda$ is the Specht module $S^\lambda$, which is spanned by the polytabloids in $M^\lambda$. We shall take a slightly unusual step here, and define the Specht module as a special case in a larger family of submodules. We will specialise here to two part partitions, but the general definitions can be found in James' book \cite{James}.

Let $t$ be a $(m-i,i)$-tableau and let $j\le i$. The $j$-\textit{column stabiliser} of $t$, denoted $C_j(t)$ is the set of permutations that fix all but the first $j$ rows of $t$ and only permute elements that appear in the same column of $t$. In particular $C_j(t)$ is generated by the $j$ transpositions which swap an element of the first $j$ entries in the first row of $t$ with the element appearing below it

The $j$-\textit{column symmetriser} is the element of the group algebra $$\kappa_j(t):=\sum_{\sigma\in C_j(t)}(-1)^\sigma\sigma,$$ and the $j$-\textit{polytabloid} is 
$$
e^j_t=\kappa_j(t)\{t\}\in M^{(m-i,i)}.
$$
We define the $k\Sn$-module $S^{(m-i,j)(m-i,i)}\subseteq M^{(m-i,i)}$ to be the submodule spanned by the $j$-polytabloids. The Specht-module $S^{(m-i,i)}$ is just the module obtained when $j=i$. 
This gives us a chain of submodules 
$$M^{(m-i,i)}\supseteq S^{(m-i,1)(m-i,i)}\supseteq \cdots S^{(m-i,i-1)(m-i,i)}\supseteq S^{(m-i,i)}\supseteq 0.$$
It can be shown that the successive quotients 
$S^{(m-i,j)(m-i,i)}/S^{(m-i,j+1)(m-i,i)}$ are isomorphic as $k\Sn$-modules to the Specht module $S^{(m-j,j)}$, and so we write 
$$M^{(m-i,i)}\sim\begin{array}{c}
    S^{(m)}   \\
    S^{(m-1,1)} \\
    \vdots \\
    S^{(m-i,i)}
\end{array}$$
 to indicate that $M^{(m-i,i)}$ has a chain of submodules whose successive quotients are $S^{(m)},S^{(m-1,1)},\dots,S^{(m-i+1,i-1)}$ and $S^{(m-i,i)}$. 
To see that these quotients are indeed the Specht modules we follow \cite{James 77} and define a $k\Sn$-homomorphism 
$$\psi_j:M^{(m-i,i)}\to M^{(m-j,j)}$$
by 
$$\psi_j(X)=\sum_{Z\subseteq_j X} Z,$$
for $X\subseteq_i [m]$. In the notation of tabloids 
$$
\psi_j(\{t\})=\sum_{\{s\}} \{s\},
$$
where the sum is over all the $(m-j,j)$-tabloids $\{s\}$ whose second row is a subset of the second row of $\{t\}$.
\begin{prop}\label{psi}
Let $t$ be an $(m-i,i)$-tableau and $e^j_t$ its $j$-polytabloid. Let $k<j\le i$ then $$\psi_k(e^j_t)=0,$$
while 
$$\psi_j(e^j_t)=e^j_{t'},$$
where $t'$ is the $(m-j,j)$-tableau obtained from $t$ by moving the last $i-j$ entries in the bottom row to the end of the top row. 
\end{prop}
\begin{proof}
\begin{align*}
\psi_k(e^j_t)&=\psi_k(\kappa_j(t)\{t\})\\
             &=\psi_k(\sum_{\sigma\in C_j(t)}(-1)^\sigma\sigma\{t\}) \\
             &=\sum_{\sigma\in C_j(t)}(-1)^\sigma\sigma\psi_k(\{t\}) \\
             &=\sum_{\sigma\in C_j(t)}(-1)^\sigma\sigma \sum_{\{s\}}\{s\},
\end{align*}
where the second sum is over all $(m-k,k)$-tabloids, $\{s\}$, whose second row is a subset of the second row of $\{t\} $. As $k<j$ then some element in the first $j$ entries of the second row of $t$ must lie in the top row of $\{s\}$, and there must be a transposition $\sigma\in C_j(t)$ fixing $\{s\}$. The terms involving this transposition and the terms not involving this transposition have opposite signs, and cancel, thus $$\psi_k(e^j_t)=0.$$
Similarly 
$$\psi_j(e^j_t)=\sum_{\sigma\in C_j(t)}(-1)^\sigma\sigma \sum_{\{s\}}\{s\},$$
with cancellation for any $\{s\}$ for which one of the first $j$ elements of the second row of $t$ appears in the first row. Therefore
\begin{align*}
\psi_j(e^j_t)&=\sum_{\sigma\in C_j(t)}(-1)^\sigma\sigma \{t'\}\\
&=\sum_{\sigma\in C_j(t')}(-1)^\sigma\sigma \{t'\}\\
&=e^j_{t'}.
\end{align*}
\end{proof}

So, when restricted to $S^{(m-i,j)(m-i,i)}\subseteq M^{(m-i,i)}$ the image of $\psi_j$ is isomorphic to $S^{(m-j,j)}$ and its kernel is $S^{(m-i,j)(m-i,i)}$. This gives an alternative characterisation of 
$S^{(m-i,j)(m-i,i)}$ as:
$$
S^{(m-i,j)(m-i,i)}=\cap_{k=0}^{j-1}(\text{ker}(\psi_k:M^{(m-i,i)}\to M^{(m-k,k)})).
$$

To prove \Cref{Rank} in the next section we will identify a submodule of $M^{(m-i,i)}$ and study its images under this map. We shall conclude this section by stating a special case of the famous hook length formula \cite{hook}, which gives the dimension of a Specht module.
\begin{theorem}\label{dim}
Let $k$ be a field. 
$$\text{dim}S^{(m-j,j)}={m \choose j}-{m \choose j-1},$$
where ${m \choose -1}=0.$
\end{theorem}
The astute reader will have noticed that is the term which appears in the formula in \Cref{Rank}. 

\section{Proof of Theorem 1}

The rows of the inclusion matrix $A_{i}^n(m)$ are indexed by the $i$ subsets of $[m]$, which is naturally the basis for the $k\Sm$-module $M^{(m-i,i)}$. The columns of span a submodule of $M^{(m-i,i)}$ of dimension $\text{rank}_k(M^{(m-i,i)})$. We shall denote this submodule by $\Pin$. Our analysis of $M^{(m-i,i)}$ in the previous section gives rise to a chain of submodules 
$$\Pin\supseteq P_i^n(m)_1\supseteq P_i^n(m)_2\supseteq\cdots\supset P_i^n(m)_i\supseteq 0,$$
where $\Pinj:=\Pin\cap S^{(m-i,j)(m-i,i)}$.
This shows that 
$$\Pin\sim
\begin{array}{c}
    L^{(m)}   \\
    L^{(m-1,1)} \\
    \vdots \\
    L^{(m-i,i)}
\end{array},
$$
where each $L^{(m-j,j)}$ is some submodule of $S^{(m-j,j)}$. In particular, $L^{(m-j,j)}$ is the image of $\Pinj$ under the map $\psi_j:M^{(m-i,i)}\to M^{(m-j,j)}$.

Observe that the columns of the matrix $A_i^n(m)$ correspond to the images of the $n$-subsets of $[m]$ under the homomorphism $\psi_i:M^{(m-n,n)}\to M^{(m-i,i)}$. The module $\Pin$ is thus the image of $\psi_i:M^{(m-n,n)}\to M^{(m-i,i)}$. Observe that although $(m-n,n)$ may not be a partition, we can still define the permutation module $M^{(m-n,n)}$, and we can also define $j$-polytabloids for any $j<\min\{m-n,n\}$. Denote by $x\in \Pin$ the image of the $j$-polytabloid corresponding to a $(m-n,n)$-tableau $t$, for $j\le i$. Our assumption from the introduction that $i<\min\{m-n,n\}$ ensures that this $j$-polytabloid is well-defined.
\begin{align*}
x:&=\psi_i(e^j_t)\\
&=\sum_{\sigma\in C_j(t)}(-1)^\sigma\sigma \sum_{\{s\}}\{s\}
\end{align*}
where the second sum is over all $(m-i,i)$-tabloids, $\{s\}$, whose second row is a subset of the second row of $\{t\}$. Of course we have cancellation of any terms for which the first $j$ entries from the second row of $t$ do not appear in the second row of $\{s\}$, so the sum is over all $(m-i,i)$-tabloids whose second row is a subset of the second row of $\{t\}$ of size $i$ containing these first $j$ entries.
Observe then that 
$$x=\sum_{s}e^j_s,$$
where the sum is over all $(m-i,i)$-tableaux obtained by moving $n-i$ of the last $n-j$ entries of the second row of $t$ to the top row.
As $x$ is a sum of $j$-polytabloids, $x\in S^{(n-i,j)(m-i,i)}$ and thus $x \in \Pinj$, and so its image under $\psi_j$ is in $L^{(m-j,j)}$.
\begin{prop}\label{whole module}
If $p\nmid {n-j \choose i-j}$ then $L^{(m-j,j)}=S^{(m-j,j)}$.
\end{prop}
\begin{proof}
Let $x=\sum_{s}e^j_s$ as above. Then, by \Cref{psi},
\begin{align*}
\psi_j(x)&= \psi_j(\sum_{s}e^j_s)\\
&= \sum_{s} e^j_{s'}
\end{align*}

where $s'$ is the $(m-j,j)$-tableau obtained from $s$ by moving the last $i-j$ entries of the bottom row to the top row. Each term is equal, and and the sum is over all $(m-i,i)$-tableaux obtained by moving $n-i$ of the last $n-j$ entries of the second row of $t$ to the top row, of which there are ${n-j \choose i-j}$. Thus 
$$\psi_j(x)={n-j \choose i-j} e^j_{s'},$$
and hence $e^j_{s'}\in L^{(m-j,j)}$. In fact, this shows that any $j$-polytabloid is in $L^{(m-j,j)}$ and thus $L^{(m-j,j)}=S^{(m-j,j)}$.
\end{proof}

\begin{proof}[Proof of Theorem 1]
Consider the image of $\Pinj$ under the map $\psi_j$, which is a submodule $L^{(m-j,j)}\subseteq S^{(m-j,j)}$. By \Cref{whole module}, if  $p\nmid {n-j \choose i-j}$ then $L^{(m-j,j)}= S^{(m-j,j)}$. On the other hand, if $p\mid {n-j \choose i-j}$ then for any column $y=\psi_i(Y)$ of $A_i^n(m)$, 
\begin{align*}
\psi_j(y)&=\sum_{X\subset_i Y}\psi_j(X)\\
&=\sum_{X\subseteq_i Y}\sum_{Z\subseteq_j X} Z\\
&=\sum_{Z\subseteq_j Y}{n-j\choose i-j} Z \\
&=0.
\end{align*}
This means that $\psi_j$ is the zero map on $\Pin$, thus $L^{(m-j,j)}=0$.
We conclude that 
$$\Pin\sim
\begin{array}{c}
    L^{(m)}   \\
    L^{(m-1,1)} \\
    \vdots \\
    L^{(m-i,i)}
\end{array},
$$
with $$L^{(m-j,j)} = \begin{cases}
0  & \text{if }  p\mid {n-j \choose i-j}\\
S^{(m-j,j)} &\text{if }  p\nmid {n-j \choose i-j}
\end{cases}.
$$
The dimension of this module is then: 
\begin{align*}
\text{dim}_k(\Pin)&=\sum_{j=0}^m \text{dim}_k(L^{(m-j,j)}))\\
&=\sum_{p\nmid {n-j \choose i-j}} \text{dim}_k(S^{(m-j,j)}))\\
&=\sum_{p\nmid {n-j \choose i-j}} {m \choose j}-{m \choose j-1},
\end{align*}
where the last equality is due to \Cref{dim}. The rank of the inclusion matrix $A_i^n(m)$ over $k$ is the dimension of the $k\Sn$-module $\Pin$, thus proving the result.
\end{proof}
\section*{Aknowledgements}
This work will appear in the author's PhD thesis prepared at the University of Cambridge and supported by the Woolf Fisher Trust and the Cambridge Trust. This work was done while the author was a visiting scholar at Victoria University of Wellington. The author would like to thank Dr Stuart Martin for his encouragement and support and Freddie Illingworth for conversations on the history of this problem.

\newpage
\thispagestyle{footer}

\end{document}